\documentclass[final%authoryear,preprint,review%,12pt
]{elsarticle-mine}

\RequirePackage[abspath,realmainfile]{currfile}

\usepackage{amsmath,amssymb,amsthm,amsfonts,amsbsy}
\usepackage{graphicx}
 
\usepackage{enumerate}
\usepackage[final]
{showkeys}

\usepackage{datetime} 
\usepackage{xcolor}

\usepackage{mathrsfs}

\usepackage{hyperref}

\theoremstyle{plain} 
\newtheorem{theorem}{Theorem}%[section]

{Lemma}

\theoremstyle{definition} 

\theoremstyle{definition} 

\newtheorem*{ex*}{Example}
\theoremstyle{remark} 

\theoremstyle{remark} 
\newtheorem{remark}[theorem]{Remark}
\newtheorem*{remark*}{Remark}

\providecommand{\url}[1]{#1}

\renewcommand{\le}{\leqslant}
\renewcommand{\ge}{\geqslant}

\newcommand{\E}{\operatorname{\mathsf{E}}}

\newcommand{\ii}[1]{\operatorname{\mathsf{I}}\{#1\}}

\newcommand{\R}{\mathbb{R}}

\newcommand{\de}{\delta}
\newcommand{\De}{\Delta}

\newcommand{\vp}{\varepsilon}

\newcommand{\ka}{\kappa}

\newcommand{\la}{\lambda}
\newcommand{\La}{\Lambda}

\newcommand{\x}{{\mathbf{x}}}
\newcommand{\y}{{\mathbf{y}}}

\newcommand{\ten}{\mathbf{10}}

\newcommand{\hf}{\hat f}

\newcommand{\NW}{\mathsf{\,NW}}
\newcommand{\PC}{\mathsf{\,PC}}
\newcommand{\GM}{\mathsf{\,GM}}

\newcommand{\CW}{\mathsf{\,CW}}

\numberwithin{equation}{section}

\begin{document}

\begin{frontmatter}


\title{Monotonicity preservation properties of kernel regression estimators}
%On the characteristic functions of the positive part and absolute value of a random variable
%\tnoteref{label1}
%
%\tnotetext[label1]{Supported by NSA grant H98230-12-1-0237}

%\title{}


\author{Iosif Pinelis}

\address{Department of Mathematical Sciences\\
Michigan Technological University\\
Houghton, Michigan 49931, USA\\
E-mail: ipinelis@mtu.edu}

\begin{abstract}
Three common classes of kernel regression estimators are considered: the Nadaraya--Watson (NW) estimator, the Priestley--Chao (PC) estimator, and the Gasser--M\"uller (GM) estimator. 
It is shown that (i) the GM estimator has a certain monotonicity preservation property for any kernel $K$, (ii) the NW estimator has this property if and only the kernel $K$ is log concave, and (iii) the PC estimator does not have this property for any kernel $K$. 
Other related properties of these regression estimators are discussed. 
\end{abstract}


%	62G05  	Nonparametric estimation
%	62G08  	Nonparametric regression and quantile regression

\end{frontmatter}

% \linenumbers

%\tableofcontents

\section{Introduction,
summary, and discussion}
\label{intro}

%two groups -- an $x$-group and a $y$-group, each consisting of $n$ individuals. We order the individuals in each group according to the values of a certain numerical characteristic, matching the individual in the $x$-group with the $i$th smallest value $x_i$ of the characteristic to the individual in the $y$-group with the $i$th smallest value $y_i$ of the characteristic, so that conditions \eqref{eq:x incr} and \eqref{eq:y incr} hold. 

We are given points $(x_1,y_1),\dots,(x_n,y_n)$ in $\R^2$. These points may be thought of as particular realizations of random pairs $(X_1,Y_1),\dots,(X_n,Y_n)$. In particular, this includes nonlinear regression models of the form 
\begin{equation}\label{eq:regr}
	Y_i=f(X_i)+\vp_i
\end{equation}
for $i\in[n]:=\{1,\dots,n\}$, where $f$ is a somewhat smooth unknown function from $\R$ to $\R$ and the $\vp_i$'s are random variables such that $\E(\vp_i|X_1,\dots,X_n)=0$ for all $i$. 

One then wants to obtain an estimator $\hf$ of the unknown function $f$. A way to do that is to smooth the data $(x_1,y_1),\dots,(x_n,y_n)$ using a kernel $K$, which is understood as a probability density function (pdf) on $\R$ -- that is, a nonnegative measurable function from $\R$ to $\R$ such that $\int_\R K(u)\,du=1$. The resulting kernel smoothers $\hf$ of the data are called kernel regression estimators. 

The kernel $K$ is usually taken according to the formula 
\begin{equation}\label{eq:K_h}
	K(u)=K_{\ka,h}(u):=\frac1h\,\ka\Big(\frac uh\Big)
\end{equation}
for real $u$, where $\ka$ can be thought of as a fixed kernel, and then $h$ is a positive real number referred to as the bandwidth, whose %optimal 
choice may depend on the model, the estimator used, the sample size $n$, and possibly on the data as well; the choice of the ``mother'' kernel $\ka$ may depend on the model and the estimator. See e.g.\ \cite{%schucany,
chu-marron}. 

Let 
\begin{equation*}
	\x:=(x_1,\dots,x_n)\quad\text{and}\quad \y:=(y_1,\dots,y_n). 
\end{equation*}
The three most common kernel regression estimators are as follows. 

The \emph{Nadaraya--Watson (NW) estimator} \cite{nadaraya,watson64} is defined by the formula 
\begin{equation}\label{eq:NW}
	\hf^\NW_K(x):=\hf^\NW_{K;\x,\y}(x)
	:=\frac{\sum_{i=1}^n{y_i\,K(x-x_i)}}{\sum_{i=1}^n{K(x-x_i)}}
\end{equation}
for all real $x$ such that the denominator $\sum_{i=1}^n{K(x-x_i)}$ of the ratio in \eqref{eq:NW} is nonzero; let us denote the set of all such $x$ by $D^\NW_{K;\x,\y}$:  
\begin{equation*}
	D^\NW_{K;\x,\y}:=\Big\{x\in\R\colon\sum_{i=1}^n{K(x-x_i)}>0\Big\}.
\end{equation*}
For $x\notin D^\NW_{K;\x,\y}$, the value of $\hf^\NW_{K;\x,\y}(x)$ is left undefined. So, 
$D^\NW_{K;\x,\y}$ is the domain (of definition) of the NW estimator $\hf^\NW_{K;\x,\y}$. 
%\begin{equation*}
%	\x:=(x_1,\dots,x_n)\quad\text{and}\quad \y:=(y_1,\dots,y_n). 
%\end{equation*}
%Here, to avoid the vanishing of the denominator, it is assumed that $K>0$ on $\R$.  

The \emph{Priestley--Chao (PC) estimator} \cite{priestley-chao} is defined by the formula 
\begin{equation}\label{eq:PC}
	\hf^\PC_K(x):=\hf^\PC_{K;\x,\y}(x):=\sum_{i=1}^n y_i\,(x_i-x_{i-1}) K(x-x_i) 
\end{equation}
for all real $x$. Here, it is assumed that the $x_i$'s are in the order of their indices, so that 
\begin{equation}\label{eq:x incr}
	x_1\le\cdots\le x_n
\end{equation}
and that $x_0$ is a real number such that $x_0\le x_1$. 

The \emph{Gasser--M\"uller (GM) estimator} \cite{gasser-muller79} is defined by the formula 
\begin{equation}\label{eq:GM}
	\hf^\GM_K(x):=\hf^\GM_{K;\x,\y}(x):=\sum_{i=1}^n y_i \int_{s_{i-1}}^{s_i}K(x-t)\,dt 
\end{equation}
for all real $x$, where 
\begin{equation*}
	s_i:=(x_i+x_{i+1})/2. 
\end{equation*}
Here, \eqref{eq:x incr} is assumed again, with the additional assumptions $x_0:=-\infty$ and $x_{n+1}:=\infty$, so that 
\begin{equation*}
\text{$s_0=-\infty$\quad and\quad $s_n=\infty$.}	
\end{equation*}
Note that $x_0$ here is not the same as $x_0$ for the PC estimator. 

The PC and GM estimators are defined on the entire real line $\R$, which is thus 
the domain of these two estimators. 

The question considered in the present note is this: 
\begin{itemize}
	\item Under what conditions on the kernel $K$ do the NW, PC, and GM kernel estimators preserve the monotonicity?
\end{itemize}

More specifically, assume that condition \eqref{eq:x incr} holds, as well as the condition 
\begin{equation}\label{eq:y incr}
	y_1\le\cdots\le y_n, 
\end{equation}
so that, if $x_i<x_j$ for some $i$ and $j$ in $[n]$, then $y_i\le y_j$. One can also say that $\x=(x_1,\dots,x_n)$ and $\y=(y_1,\dots,y_n)$ are co-monotone. 
%IP05-12-21
The co-monotonicity condition will be discussed in Section~\ref{discuss}. 

%In view of the possibility of an appropriate re-enumeration, one may understand the co-monotonicity condition slightly more generally (without assuming \eqref{eq:x incr} or \eqref{eq:y incr}), as the condition that 
%\begin{equation*}
%	(x_i-x_j)(y_i-y_j)\ge0
%\end{equation*}
%for all $i$ and $j$ in the set $\{1,\dots,n\}$. 

%IP05-12-21
%The co-monotonicty condition arises naturally e.g.\ in the following setting: We have 
%two groups -- labeled, say, as an $x$-group and a $y$-group, each consisting of $n$ individuals. The individuals in each group are ordered according to the values of a certain numerical characteristic, matching the individual in the $x$-group with the $i$th smallest value $x_i$ of the characteristic to the individual in the $y$-group with the $i$th smallest value $y_i$ of the characteristic, so that conditions \eqref{eq:x incr} and \eqref{eq:y incr} hold. 

Let us say that the NW kernel estimator preserves the monotonicity for a given kernel $K$ if the function $\hf^\NW_K=\hf^\NW_{K;\x,\y}$ is nondecreasing (on its domain $D^\NW_{K;\x,\y}$) for any natural $n$ and any co-monotone $\x$ and $\y$ in $\R^n$. Similarly defined are the monotonicity preservation properties for the PC and GM kernel estimators, with the domain $D^\NW_{K;\x,\y}$ of course replaced by $\R$ for the latter two estimators. 

%IP05-12-21
%The monotonicity preservation property appears to be natural and desirable for a curve estimator. This point is an instance of the general principle that it is desirable for the values of a statistical estimator to be in the set of all values of the estimated function of the unknown distribution. E.g., it is natural to want 
%the values of an estimator of a nonnegative parameter to be nonnegative; 
%the values of an estimator of a pdf to be pdf's; etc. 
%
%As pointed out e.g.\ in \cite{hall01}, ``Monotone estimates are of course required in many practical applications, where physical considerations suggest that a response should be monotone in the dosage or the explanatory variable.'' Various methods for monotonizing kernel estimators have been proposed, including the isotonic regression methods using constrained optimization \cite{fried-tibsh}; 
%a method based on the minimization of misclassification costs
%\cite{bloch-silver}; 
%constrained spline-based methods \cite{mammen--t-a99}
%and other projection techniques \cite{mammen-etal01}; 
%ones based on maximizing fidelity to the conventional empirical approach subject to monotonicity \cite{hall01}; 
%monotone smoothing by inversion \cite{dette06}; 
%a method based on the Hardy--Littlewood--P\'olya monotone rearrangement \cite{anevski-fougeres}. 
%The paper 
%\cite{chetv-wilh} uses a sieve (rather than kernel) estimator which imposes the monotonicity constraint. 

%the monotonicity of the regression estimator is required in many applications of statistics; see also further references in \cite{hall01,dette06}. 

The main results of this note, which will be proved in Section~\ref{proofs}, are Theorems~\ref{prop:NW}, \ref{prop:PC}, and \ref{prop:GM}, which 
characterize the kernels $K$ for which the NW, PC, and GM kernel estimators preserve the monotonicity.  

\begin{theorem}\label{prop:NW}
%Suppose that the kernel $K$ is strictly positive. 
%and has the right derivative $K':=K'_+$ on $\R$. 
%Then the 
The NW kernel estimator preserves the monotonicity for a given kernel $K$ if and only if $K$ is log concave. 
\end{theorem}

Recall here that a nonnegative function $g$ is log concave if $\ln g$ is concave, with $\ln0:=-\infty$. An important example of a log-concave kernel is any normal pdf. Also, if $K=K_{\ka,h}$ is as in \eqref{eq:K_h} with $\ka(u)=c_p\,e^{-|u|^p}$ for some real $p\ge1$ and all real $u$ (with $c_p:=1/\int_{-\infty}^\infty e^{-|u|^p}\,du$) or with  $\ka(u)=b\,e^{-1/(1-u^2)}\ii{|u|<1}$ for all real $u$ (with $b:=1/\int_{-1}^1 e^{-1/(1-u^2)}\,du$ and $\ii\cdot$ denoting the indicator), then $K$ is log concave. 
Also, the arbitrarily shifted and rescaled pdf's of the gamma distribution with shape parameter $\ge1$ and of the beta distribution with both parameters $\ge1$ are log concave. 
It is easy to see that the tails of any log-concave kernel $K$ necessarily decrease at least exponentially fast. Also, clearly the kernel $K_{\ka,h}$ defined by \eqref{eq:K_h} is log concave for each real $h>0$ if the corresponding mother kernel $\ka$ is log concave. 

In a somewhat more specific setting, the ``if'' part of Theorem~\ref{prop:NW} was essentially presented in \cite[Remark~2.1]{mukerjee}, based on a monotone likelihood ratio property of a posterior distribution, with a reference to \cite[Lemma~2, page~74]{lehmann59-test}. However, the latter lemma does not explicitly mention a posterior distribution. Therefore, we shall give a short, direct, and self-contained proof of the ``if'' part of Theorem~\ref{prop:NW}, which will also be used to prove the ``only if'' part of Theorem~\ref{prop:NW}. 

\begin{theorem}\label{prop:PC} 
The PC kernel estimator does not preserve the monotonicity for any given kernel $K$. More specifically, for any kernel $K$, any natural $n$, and any co-monotone $\x$ and $\y$ in $\R^n$, the function $\hf^\PC_{K;\x,\y}$ is not nondecreasing -- unless $\x$ and $\y$ are trivial in the sense that 
\begin{equation}\label{eq:trivial}
\text{$y_i\,(x_i-x_{i-1})=0$ for all $i\in[n]$ }	
\end{equation}
(in which case $\hf^\PC_{K;\x,\y}$ is identically $0$). 
\end{theorem}

\begin{theorem}\label{prop:GM} 
The GM kernel estimator preserves the monotonicity for any given kernel $K$. 
\end{theorem}

\begin{remark}
It immediately follows from the definitions \eqref{eq:NW}, \eqref{eq:PC}, and \eqref{eq:GM} that the NW, PC, and GM kernel estimators are linear in $\y$. In particular, if $\y$ is replaced by $-\y$, then the values of these estimators change to their opposites. Therefore, the monotonicity preservation property of any one of these three estimators implies the corresponding constancy preservation property, by which we mean the following: if $y_1=\dots=y_n$, then the corresponding values of the estimators %value $\hf^\NW_K(x)$, $\hf^\PC_K(x)$, $\hf^\GM_K(x)$ does 
do not depend on $x$. 
%; moreover, this constant value is $c$. 
%So, the constancy preservation property is weaker than the monotonicity preservation property. 

In fact, it is obvious that, for any kernel $K$, the NW and GM kernel estimators have the constancy preservation property; 
moreover, they have the constant preservation property: if $y_1=\dots=y_n=c$, then the NW and GM kernel estimators have the constant value $c$. On the other hand, in view of Theorem~\ref{prop:PC},  for any kernel $K$, the function $\hf^\PC_{K;\x,\y}$ for $y_1=\dots=y_n$ is constant 
if and only if at least one of the following two trivial cases takes place: (i) $y_1=\dots=y_n=0$ or (ii) $x_1=\dots=x_n$. 

The NW and GM kernel estimators also have the shift preservation property (which actually follows from the constant preservation property and the linearity): If $y_1,\dots,y_n$ are replaced by $y_1+c,\dots,y_n+c$ for some real $c$, then $\hf^\NW_{K;\x,\y}$ and $\hf^\GM_{K;\x,\y}$ are replaced by $\hf^\NW_{K;\x,\y}+c$ and $\hf^\GM_{K;\x,\y}+c$, respectively. On the other hand, for any kernel $K$, if $y_1,\dots,y_n$ are replaced by $y_1+c,\dots,y_n+c$ for some real $c$, then $\hf^\PC_{K;\x,\y}$ is replaced by $\hf^\PC_{K;\x,\y}+c$
if and only if at least one of the following two trivial cases takes place: (i) $c=0$ or (ii) $x_1=\dots=x_n$.  

Summarizing this remark, we may say that the NW and GM kernel estimators always have the  constancy and shift preservation properties, whereas the PC estimator practically never has these nice properties. 
\end{remark}

The presence -- or, in the case of the PC estimator, absence -- of the monotonicity and shift preservation properties is illustrated in Figure~\ref{fig:1}. % and \ref{fig:2}. 

The upper row in Figure~\ref{fig:1} shows graphs of $\hf^\NW_{K;\x,\y}$, $\hf^\PC_{K;\x,\y}$, and $\hf^\GM_{K;\x,\y}$ for the (randomly generated) $20$-tuples
% \nadaraya-watson-mono\4.nb
\begin{multline}
	\x=(-8.8, -8, -6.8, -6.3, -4.3, -3.9, -3.9, -3.7, -2.8, \\ 
	-2, -1.8, -1, 
-1, %\\ 
1.9, 2.3, 2.9, 5.2, 6.5, 9.3, 10) \label{eq:xx}
\end{multline}
and 
\begin{multline}
	\y=(-7.1, -6.1, -5.8, -5.3, -4.9, -1.2, -0.6, -0.4, 0.8, \\ 
	2, 2.1, 2.3, 
2.4, 4.4, %\\ 
5.9, 6, 6.9, 7.4, 8.1, 9.3),  \label{eq:yy}
\end{multline}
with $K$ of the form $K_h:=K_{\ka,h}$ as in \eqref{eq:K_h} with $\ka$ being the standard normal density. 

For each of the three graphs in the upper row of Figure~\ref{fig:1}, the bandwidth $h$ of the kernel $K_h$ is determined by cross-validation (see e.g.\ \cite{stone74}) -- that is, as an approximate minimizer (obtained numerically) of 
\begin{equation*}
	\CW(h):=\sum_{j=1}^{20}\big(y_j-\hf_{K_h;\x^{(j)},\y^{(j)}}(x_j)\big)^2
\end{equation*}
in real $h>0$, where $\hf\in\{\hf^\NW,\hf^\PC,\hf^\GM\}$, 
$\x^{(j)}:=(x_1,\dots,x_{j-1},x_{j+1},\dots,x_{20})$, and $\y^{(j)}:=(y_1,\dots,y_{j-1},y_{j+1},\dots,y_{20})$. 

The lower row in Figure~\ref{fig:1} shows graphs of 
$\hf^\NW_{K;\x,\y+\ten}$, $\hf^\PC_{K;\x,\y+\ten}$, and $\hf^\GM_{K;\x,\y+\ten}$ for $\x$ and $\y$ as in \eqref{eq:xx} and \eqref{eq:yy}, with $\y$ replaced by its shifted version $\y+\ten:=(y_1+10,\dots,y_n+10)$. Here the kernel $K$ is of the same form as the one used for the upper row of Figure~\ref{fig:1}, with the bandwidth $h$ still determined by cross-validation. 

For the graphs of $\hf^\PC_{K;\x,\y}$ and $\hf^\PC_{K;\x,\y+\ten}$ in Figure~\ref{fig:1}, it is assumed that $x_0:=x_1-h$. 

The corresponding data points are also shown in Figure~\ref{fig:1}. 

Figure~\ref{fig:1} illustrates the monotonicity and shift preservation properties of the NW and GM estimators and the lack of these properties for 
the PC estimator. 

One may also note that the GM graphs in Figure~\ref{fig:1} look very similar to the corresponding NW graphs. However, other choices of $\x$ and $\y$ suggest that the GM graphs are usually a bit smoother than the NW ones. 

A possible reason for this is that the cross-validation quality $\CW(h)$ for the NW estimator is usually rather flat (that is, almost constant) in a large neighborhood of a minimizer $h$ of $\CW(h)$, and hence the choice of a numerical minimizer $h$ of $\CW(h)$ for the NW estimator may be rather unstable, which can then affect the smoothness of the resulting NW estimator. Moreover, $\CW$ is usually rather flat for the GM estimator as well. This observation is illustrated in Figure~\ref{fig:2}. 

On the other hand, as clearly seen from the definition \eqref{eq:GM}, the GM estimator is always continuous, for any, however discontinuous, kernel $K$. Of course, this cannot be said concerning the NW and PC estimators. 

Figures~\ref{fig:1} and \ref{fig:2} also suggest that, at least for the co-monotone $\x$ and $\y$, the NW and GM curves fit the data $(\x,\y)$ significantly better than   
the PC estimator does. In particular, for $\x$ and $\y$ as in \eqref{eq:xx} and \eqref{eq:yy} the smallest values of the cross-validation quality $\CW$ are about $27.1$, $71.7$, and $20.7$ for the NW, PC, and GM estimators, respectively; here, of course, the lower values of $\CW$ correspond to the higher quality of the fit. 

Figure~\ref{fig:3} is quite similar to Figure~\ref{fig:1} except that in Figure~\ref{fig:3} for the mother kernel $\ka$ we use the ``rectangular'' uniform density on the interval $(-1/2,1/2)$, instead of the standard normal density. 

We see that the graphs of the GM estimator in Figure~\ref{fig:3} are much smoother than the corresponding graphs of the NW estimator; actually, in this setting the GM estimator is a piecewise affine continuous function. 

%In fact, as clearly seen from the definition \eqref{eq:GM}, the GM estimator is always continuous, for any, however discontinuous, kernel $K$. 

Concerning the values of the cross-validation quality $\CW$ for the ``rectangular'' kernel $K$, they were about $105.8$ for each of the two instances of the NW estimator and $23.4$ for each of the two instances of the GM estimator, whereas for first and second instances of the PC estimator the respective $\CW$ values were about $99.2$ and $2257.1$. Thus, for such a discontinuous ``rectangular'' kernel $K$, the GM estimator appears to perform much better than the NW and PC ones.  The performance of the PC estimator was especially poor in the second instance. 

Figure~\ref{fig:4} is quite similar to Figures~\ref{fig:1} and \ref{fig:3} except that in Figure~\ref{fig:4} the mother kernel $\ka$ is the infinitely smooth pdf defined by the formula 
$\ka(u)=b\,e^{-1/(1-u^2)}\ii{|u|<1}$ for all real $u$ with $b:=1/\int_{-1}^1 e^{-1/(1-u^2)}\,du$. 

We see that the graphs of the GM estimator in Figure~\ref{fig:4} look again smoother than the corresponding graphs of the NW estimator; however, in this case all the curves are infinitely smooth. 

\begin{figure}[p]
	\centering
		\includegraphics[width=1.00\textwidth]{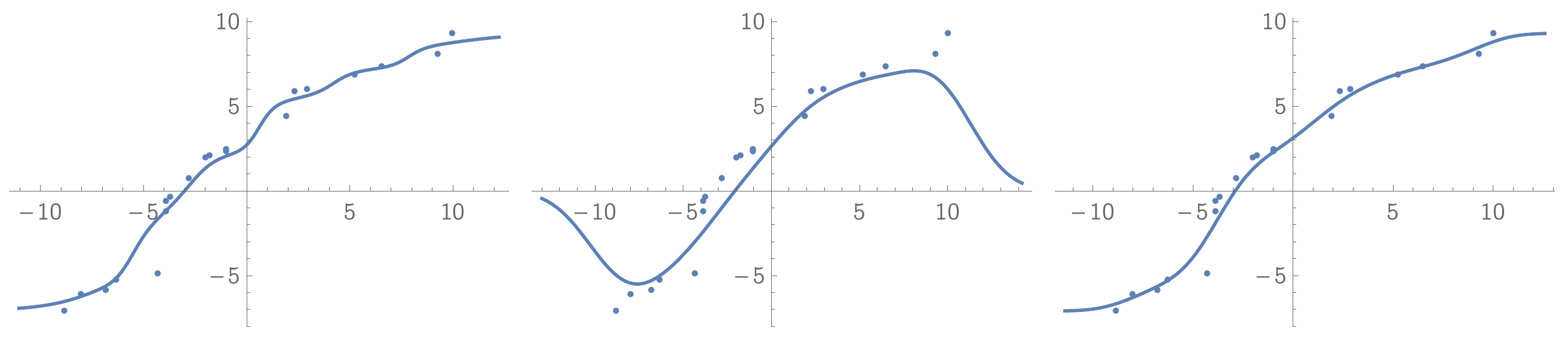}
		
		\smallskip
		
		\includegraphics[width=1.00\textwidth]{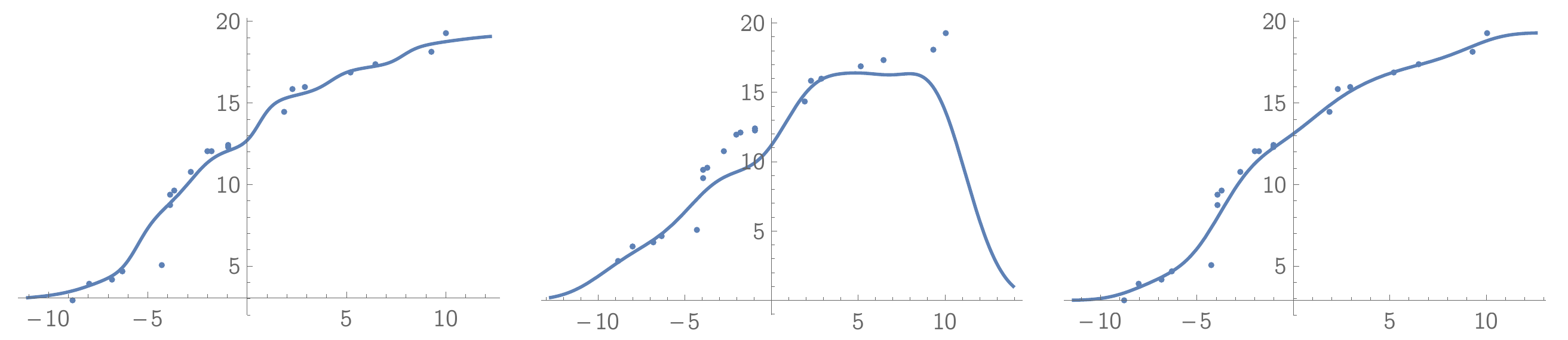}
	\caption{Upper row: graphs of $\hf^\NW_{K;\x,\y}$ (left), $\hf^\PC_{K;\x,\y}$ (middle), and $\hf^\GM_{K;\x,\y}$ (right) for $\x$ and $\y$ as in \eqref{eq:xx} and \eqref{eq:yy}. 
	Lower row: graphs of $\hf^\NW_{K;\x,\y+\ten}$ (left), $\hf^\PC_{K;\x,\y+\ten}$ (middle), and $\hf^\GM_{K;\x,\y+\ten}$ (right) for $\x$ and $\y$ as in \eqref{eq:xx} and \eqref{eq:yy}. Here $K$ is of the form \eqref{eq:K_h} with $\ka$ being the standard normal density.}
	\label{fig:1}
\end{figure}

\begin{figure}[p]
	\centering
		\includegraphics[width=1.00\textwidth]{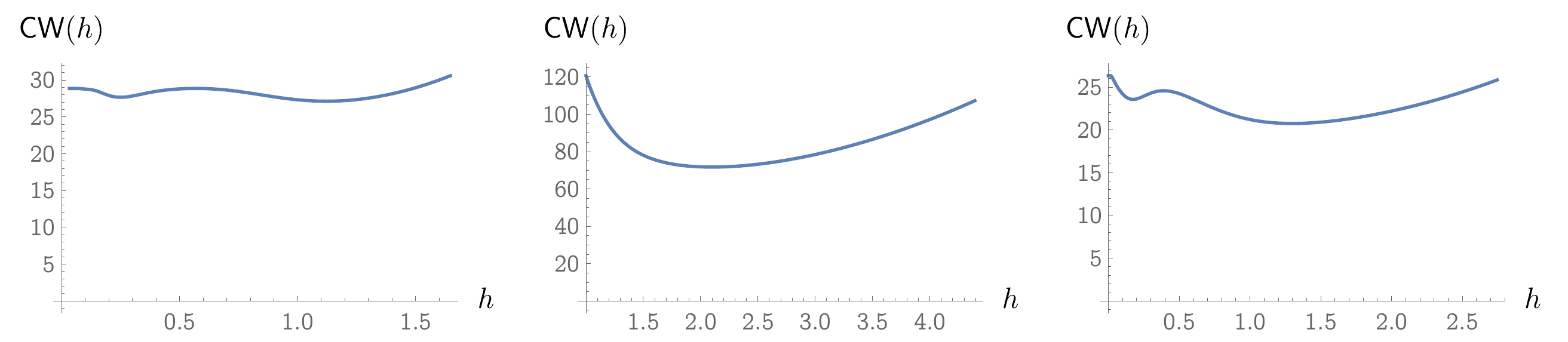}
	\caption{\rule[0pt]{0pt}{0pt} Graphs of $\CW$ for the NW estimator (left), the PC estimator (middle), and the GM estimator (right) for $\x$ and $\y$ as in \eqref{eq:xx} and \eqref{eq:yy}. Here $K$ is of the form \eqref{eq:K_h} with $\ka$ being the standard normal density.}
	\label{fig:2}
\end{figure}

\begin{figure}[p]
	\centering
		\includegraphics[width=1.00\textwidth]{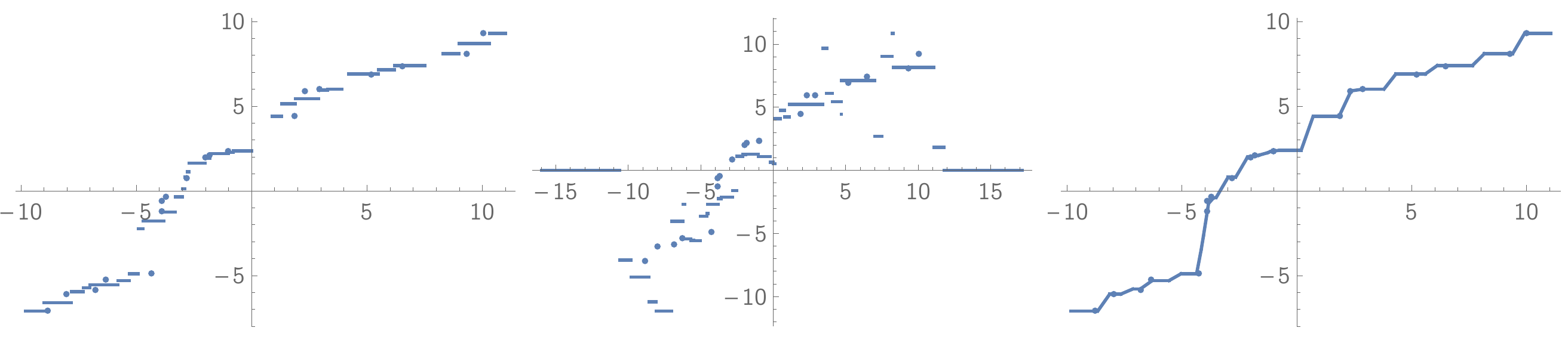}
		
		\smallskip
		
		\includegraphics[width=1.00\textwidth]{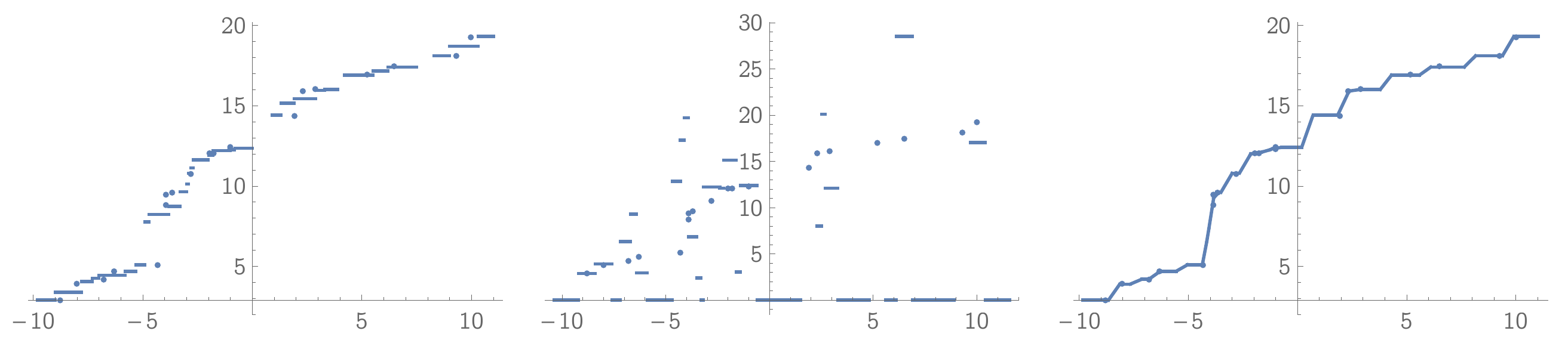}
	\caption{Upper row: graphs of $\hf^\NW_{K;\x,\y}$ (left), $\hf^\PC_{K;\x,\y}$ (middle), and $\hf^\GM_{K;\x,\y}$ (right) for $\x$ and $\y$ as in \eqref{eq:xx} and \eqref{eq:yy}. 
	Lower row: graphs of $\hf^\NW_{K;\x,\y+\ten}$ (left), $\hf^\PC_{K;\x,\y+\ten}$ (middle), and $\hf^\GM_{K;\x,\y+\ten}$ (right) for $\x$ and $\y$ as in \eqref{eq:xx} and \eqref{eq:yy}. Here $K$ is of the form \eqref{eq:K_h} with $\ka$ being the uniform density on $(-1/2,1/2)$.}
	\label{fig:3}
\end{figure}

\begin{figure}[p]
	\centering
		\includegraphics[width=1.00\textwidth]{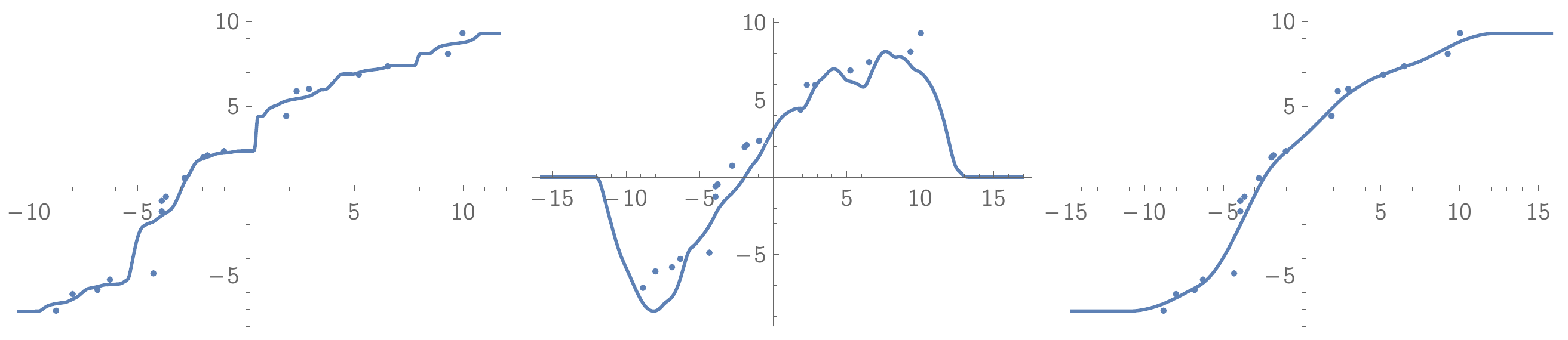}
		
		\smallskip
		
		\includegraphics[width=1.00\textwidth]{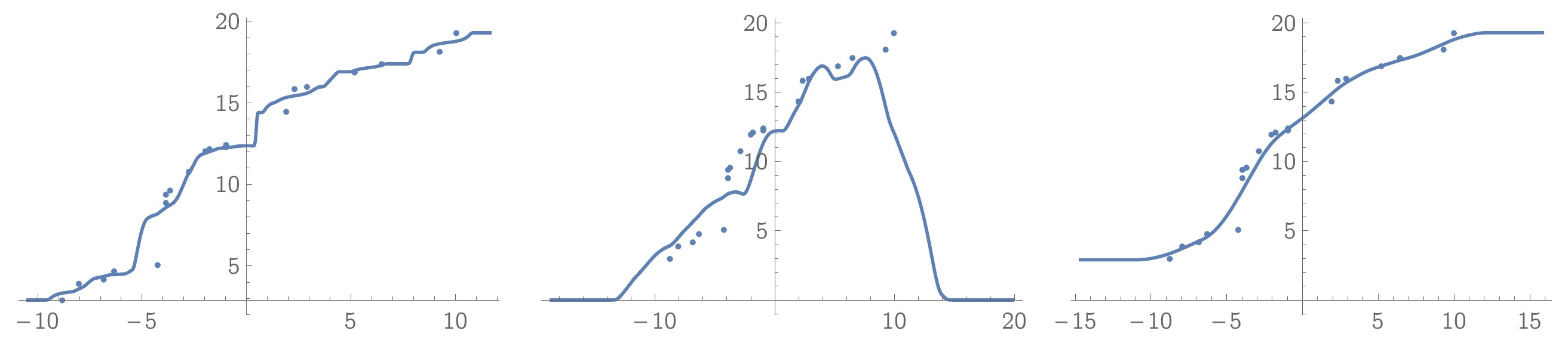}
	\caption{Upper row: graphs of $\hf^\NW_{K;\x,\y}$ (left), $\hf^\PC_{K;\x,\y}$ (middle), and $\hf^\GM_{K;\x,\y}$ (right) for $\x$ and $\y$ as in \eqref{eq:xx} and \eqref{eq:yy}. 
	Lower row: graphs of $\hf^\NW_{K;\x,\y+\ten}$ (left), $\hf^\PC_{K;\x,\y+\ten}$ (middle), and $\hf^\GM_{K;\x,\y+\ten}$ (right) for $\x$ and $\y$ as in \eqref{eq:xx} and \eqref{eq:yy}. Here $K$ is of the form \eqref{eq:K_h} with $\ka(u)=b\,e^{-1/(1-u^2)}\ii{|u|<1}$ for all real $u$, where $b=1/\int_{-1}^1 e^{-1/(1-u^2)}\,du$.}
	\label{fig:4}
\end{figure}

\newpage

\section{Discussion}\label{discuss}

The monotonicity preservation property appears to be natural and desirable for a curve estimator. This is a consequence of the general principle that it is desirable for the values of a statistical estimator to be in the set of all values of the estimated function of the unknown distribution. E.g., it is natural to want 
the values of an estimator of a nonnegative parameter to be nonnegative; 
the values of an estimator of a pdf to be pdf's; etc. 

In the rest of the section, we discuss settings where the co-monotonicity condition either holds naturally or is customarily created, and where therefore the monotonicity preservation property is of value. 

\subsection{Regression settings}\label{regr}
In regression settings, where, for each $i\in[n]$, the pair of numbers $(x_i,y_i)$ represents the observed/measured values of certain two characteristics of the $i$th individual/sample item, the co-monotonicity condition will hardly ever hold, especially when the sample size $n$ is not too small. Yet, as pointed out by Hall and Huang \cite{hall01}, even in such settings, ``Monotone estimates are of course required in many practical applications, where physical considerations suggest that a response should be monotone in the dosage or the explanatory variable.'' 

There are two kinds of methods to obtain a smooth monotone curve estimator in regression settings: 
\begin{enumerate}[(i)]
	\item the IS methods: first  \underline{i}sotonizing/monotonizing the raw regression data $\big((x_1,y_1),\dots,(x_n,y_n)\big)$ to satisfy the co-monotonicity condition and then \underline{s}moothing the monotonized data via (say) a kernel estimator; 
	\item the SI methods: first \underline{s}moothing the raw, non-monotonized regression data via (say) a kernel estimator and then \underline{i}sotonizing/monotonizing the output of the smoothing step.
\end{enumerate}
 
Mammen~\cite{mammen91} showed that, under general and rather common conditions, the IS and SI estimators are first-order equivalent to each other in terms of certain estimation errors, with the PC estimator used for the smoothing step. Somewhat similar results, in a 
more general setting and for a different kind of estimation errors were obtained by Lopuha\"{a} and Musta~\cite{lopuhaa-musta19}, also with a PC-type smoothing estimator. 
%Similar near-commutativity, IS$\,\approx\,$SI, of the I- and S-operators can be expected for the NW and GM estimators.   
%
%It should be possible to obtain similar results for the NW and GM estimators. 

However, for the IS estimator to output a monotone curve, the kernel estimator used for the second, smoothing step of the IS procedure must have the monotonicity preservation property. As shown in Theorem~\ref{prop:PC}, the PC estimator \emph{never} preserves the monotonicity. This calls for results similar to those in \cite{mammen91,lopuhaa-musta19}, but for the NW and GM estimators in place of the PC one; a near-commutativity, IS$\,\approx\,$SI, of the I- and S-operators can be expected for the NW and GM estimators as well. In fact, the first-order asymptotics of an estimation error for the IS estimator using, for the smoothing step, the NW estimator with the symmetric rectangular kernel was obtained already by Mukerjee \cite[Theorem~4.2]{mukerjee}. Related results for a current status
model were given by Groeneboom, Jongbloed, and Witte \cite{groeneboom-etal10}.

One can also note that the IS methods appear to be computationally easier, because in an IS procedure, in contrast with an SI one, the I step is applied to a \emph{finite} data set; in particular, compare formulas (2.5) and (2.6) in \cite{mammen91} for the I-steps in the IS and and SI procedures, respectively. 

Isotonization of raw data $\big((x_1,y_1),\dots,(x_n,y_n)\big)$ goes back at least to Brunk (see e.g.\ \cite{brunk-etal55}) and Grenander~\cite{grenander56}. A general framework for Grenander-type estimators, including rates of convergence, has been given recently by Westling and Carone~\cite{westling-carone}.

\subsection{Matchings}\label{match}
The co-monotonicty condition holds naturally e.g.\ in the following kind of settings: We have 
two groups -- labeled, say, as an $x$-group and a $y$-group, each consisting of $n$ individuals/items. The items in the $x$-group are ordered according to the values of a certain numerical characteristic, say the $x$-characteristic. Similarly, the items in the $y$-group are ordered according to the values of a certain numerical $y$-characteristic, possibly different from the $x$-characteristic. Then the item in the $x$-group with the $i$th smallest value $x_i$ 
of the $x$-characteristic is \emph{matched} with the $i$th smallest value $y_i$ 
of the $y$-characteristic in the $y$-group. Then, of course, the co-monotonicity conditions \eqref{eq:x incr} and \eqref{eq:y incr} will hold. 
\big(On the other hand, regression models such as \eqref{eq:regr} are not applicable in ``matching'' settings.\big) 

% with the $i$th smallest value $y_i$ 
%of the characteristic
%
%each group are ordered according to the values of certain numerical characteristics (possibly different for the two groups), \emph{matching} the item in the $x$-group with the $i$th smallest value $x_i$ 
%of the characteristic to the item in the $y$-group with the $i$th smallest value $y_i$ 
%of the characteristic, so that conditions \eqref{eq:x incr} and \eqref{eq:y incr} hold. 

For instance, %for each $i\in[n]$ 
let, as usual, $x_{n:i}$ denote the $i$th smallest value among the values $x_1,\dots,x_n$ of an iid sample from the distribution over $\R$ with a cdf $F$; that is, $x_{n:i}$ is the value of the $i$th order statistic. 
Matching then the $x_{n:i}$'s 
with the corresponding $i/n$'s, we obtain the representation $\big((x_{n:1},1/n),\dots,(x_{n:n},n/n)\big)$
of the corresponding empirical distribution function. Clearly, the co-monotonicity condition then holds: $x_{n:1}\le\dots\le x_{n:n}$ and $1/n\le\dots\le n/n$. So, smoothing the discrete graph $\big((x_{n:1},1/n),\dots,(x_{n:n},n/n)\big)$ by means of a monotonicity-preserving kernel estimator, we obtain a smooth nondecreasing estimate of the cdf $F$. 

One can similarly use the inverse discrete graph $\big((1/n,x_{n:1}),\dots,(n/n,x_{n:n})\big)$ to obtain a smooth nondecreasing estimate of the quantile function $F^{-1}$ given, say, by the formula $F^{-1}(u):=\min\{x\in\R\colon F(x)\ge u\}$ for $u\in(0,1)$. Examples of data of a form similar to that of $\big((1/n,x_{n:1}),\dots,(n/n,x_{n:n})\big)$ are, say, household income percentiles data; see e.g.\ such data for the UK~\cite{UKincome} and the US~\cite{USincome}. 

Let now $y_{n:i}$ denote the $i$th smallest value among the values $y_1,\dots,y_n$ of an iid sample from the distribution over $\R$ with a cdf $G$, possibly different from the cdf $F$. Matching then the order statistics values $x_{n:1},\dots,x_{n:n}$
with the corresponding values $y_{n:1},\dots,y_{n:n}$, we obtain the (automatically  co-monotonic) representation $\big((x_{n:1},y_{n:1}),\dots,(x_{n:n},y_{n:n})\big)$ of a Q-Q (quantile-quantile) plot for the cdf's $F$ and $G$. So, applying a monotonicity-preserving smoothing kernel estimator to this Q-Q plot, we obtain a smooth nondecreasing estimate of the ``theoretical quantile-quantile function'' $G^{-1}\circ F$. In particular, this way we can obtain a smooth monotonic plot of the US household income percentiles against the corresponding percentiles for the UK; similar graphical comparisons can of course be made between, say, different strata/categories of a population.  

Yet another example in this vein concerns point processes; see e.g.\ \cite{kallenberg83}. For a real $T>0$, let $\xi$ be a point process on the interval $[0,T]$, that is, an integer-valued (nonnegative) random measure on the Borel $\sigma$-algebra over $[0,T]$, so that $\xi=\sum_{i=1}^N\de_{X_i}$ for some nonnegative integer-valued r.v.\ $N$ and some r.v.'s $X_i$ with values in $[0,T]$, where $\de_a$ denotes the Dirac probability measure at a point $a$. Let $n,x_1,\dots,x_n$ be sample values/realizations of $N,X_1,\dots,X_n$, respectively. Then 
$x_{n:1},\dots,x_{n:n}$ can be interpreted as the consecutive times of the occurrences of the ``events'' of the point process, and  
the co-monotonicity condition will obviously hold for the data $\big((x_{n:1},1),\dots,(x_{n:n},n)\big)$. Smoothing such data by means of a monotonicity-preserving kernel estimator, we obtain a smooth nondecreasing function $\hat f$ that is an estimate of the cumulative intensity function $\La_\xi$ of the point process $\xi$, given by the formula $\La_\xi(t):=\E\xi([0,t])$ for $t\in[0,T]$; then the derivative $\hat f'$ of $\hat f$ may serve as an estimate of the intensity function $\la_\xi:=\La'_\xi$ of the point process $\xi$. Of course, to improve the quality of such estimation, instead of just one realization of the point process one can use a number of such realizations if they are available.

%\cite{UKincome,USincome}: Household Income Percentiles; similar comparisons between different population strata/categories

\section{Proofs}\label{proofs}

The proofs of Theorems~\ref{prop:NW}, \ref{prop:PC}, and \ref{prop:GM} given below are each based on quite different ideas. 

\begin{proof}[Proof of Theorem~\ref{prop:NW}]\ \\
Consider first the ``if'' part. Here we suppose that $K$ is log concave. Take any $x$ and $z$ in $D^\NW_{K;\x,\y}$ such that $x<z$. We have to show that $\hf^\NW_K(z)\ge\hf^\NW_K(x)$. Letting for brevity
\begin{equation*}
	k_i:=K(x-x_i)\quad\text{and}\quad l_i:=K(z-x_i),
\end{equation*}
$\sum_i:=\sum_{i\in[n]}$, and $\sum_{i,j}:=\sum_{i\in[n],j\in[n]}$,
we see that %for all 
\begin{equation}\label{eq:6lines}
\begin{aligned}
	&2\big(\hf^\NW_K(z)-\hf^\NW_K(x)\big)
\sum_i k_i \sum_j l_j \\ 
	&=2\Big(\frac{\sum_j y_jl_j}{\sum_j l_j}-\frac{\sum_i y_ik_i}{\sum_i k_i}\Big)
	\sum_i k_i\sum_j l_j \\ 
	&=2\sum_j y_jl_j\sum_i k_i-2\sum_i y_ik_i \sum_j l_j  \\ 	
	&=\sum_j y_jl_j\sum_i k_i+\sum_i y_il_i\sum_j k_j
	-\sum_i y_ik_i \sum_j l_j-\sum_j y_jk_j \sum_i l_i  \\ 	
	&=\sum_{i,j} (y_jl_jk_i+y_il_ik_j
	-y_ik_i l_j-y_jk_j l_i)  \\ 	
	&=\sum_{i,j} (y_j-y_i)(l_j k_i-k_j l_i).    	
\end{aligned}	
\end{equation}
For any $i$ and $j$ in $[n]$ such that $i\le j$ we have $x_i\le x_j$ and hence, by the log-concavity of $K$, $k_i\ge l_i^{1-t}k_j^t$ and $l_j\ge l_i^t k_j^{1-t}$, where  $t:=t_{i,j}:=(z-x)/(x_j-x_i+z-x)\in[0,1)$ and $0^0:=0$, so that $l_j k_i\ge k_j l_i$ and hence $(y_j-y_i)(l_j k_i-k_j l_i)\ge0$. The latter inequality similarly holds for any $i$ and $j$ in $[n]$ such that $i\ge j$. So, by \eqref{eq:6lines}, 
we have $\hf^\NW_K(z)\ge\hf^\NW_K(x)$, which completes the proof of the ``if'' part of Theorem~\ref{prop:NW}. 

Consider now the ``only if'' part of Theorem~\ref{prop:NW}. Here we are assuming that 
the NW kernel estimator preserves the monotonicity for $K$, and we have to show that $K$ is then log concave. It is enough to show that $K$ is log concave on the set $s(K):=\{x\in\R\colon K(x)>0\}$. 

Take $n=2$, $y_1=0$, $y_2=1$, $x_1=0$, $x_2=(v-u)/2$, $x=(v+u)/2$, and $z=v$ for any $u$ and $v$ in $s(K)$ such that $u<v$. Then, by \eqref{eq:6lines}, $l_2 k_1\ge k_2 l_1$, that is, $K((v+u)/2)^2\ge K(u)K(v)$, which means that $\ln K$ 
%
%
%is any real number and $v$ is any positive real number, so that $x_1<x_2$ and $x<z$. Then the same multi-line display yields 
%\begin{align*}
%	&0\le\big(\hf^\NW_K(z)-\hf^\NW_K(x)\big)\sum_i k_i \sum_j l_j
%\\	
%	&=\big(\exp\{L(u)-L(u-v)\}-\exp\{L(u+v)-L(u)\}\big)\,
%	k_1 k_2, 	
%\end{align*}
%whence $L(u)-L(u-v)\ge L(u+v)-L(u)$. So, $L$ 
is midpoint concave. Also, $\ln K$ is Lebesgue measurable, since $K$ is a pdf. By Sierpi\'nski's theorem \cite{sierp20}, 
any Lebesgue measurable midpoint concave function is concave. 
So, $\ln K$ is concave and thus $K$ is log concave. 
Now the ``only if'' part of Theorem~\ref{prop:NW} is proved as well. 
\end{proof}

\begin{proof}[Proof of Theorem~\ref{prop:PC}]
Take any kernel $K$, any natural $n$, and any co-monotone $\x$ and $\y$ in $\R^n$ such that the function $\hf^\PC_{K;\x,\y}$ is nondecreasing. We have to show that then $\x$ and $\y$ are trivial in the sense that \eqref{eq:trivial} holds. 

Since $K$ is a pdf, \eqref{eq:PC} implies
\begin{equation*}%\label{eq:PC}
	\int_\R\hf^\PC_K(x)\,dx:=\sum_{i=1}^n y_i\,(x_i-x_{i-1})\in\R, 
\end{equation*}
so that $\int_\R\hf^\PC_K\in L^1(\R)$. 

However, the only nondecreasing function $f\in L^1(\R)$ is the zero function. Indeed, if $f(a)>0$ for some $a\in\R$, then $f\ge f(a)>0$ on the interval $[a,\infty)$ and hence $\int_{[a,\infty)}f(x)\,dx=\infty$, which contradicts the assumption $f\in L^1(\R)$. Similarly, if $f(a)<0$ for some $a\in\R$, then $f\le f(a)<0$ on the interval $(-\infty,a]$ and hence $\int_{(-\infty,a]}f(x)\,dx=-\infty$, which again contradicts the assumption $f\in L^1(\R)$. 

Therefore and because the function $\hf^\PC_{K;\x,\y}$ was assumed to be nondecreasing, we conclude that $\hf^\PC_{K;\x,\y}$ must be the zero function. Recalling \eqref{eq:PC} again and applying the Fourier transform, we see that 
\begin{equation}\label{eq:=0}
	\sum_{j=1}^n y_j\,(x_j-x_{j-1})e^{itx_j} \hat K(t)=0
\end{equation}
for all real $t$, where $\hat K$ is the Fourier transform/characteristic function of the pdf $K$ given by the formula 
\begin{equation*}
	\hat K(t):=\int_\R e^{itx}K(x)\,dx  
\end{equation*}
and $i$ is the imaginary unit. 
Since $\hat K(0)=1$ and the function $\hat K$ is continuous, there exists some real $t_0>0$ such that for all $t\in(-t_0,t_0)$ we have 
$\hat K(t)\ne0$ and hence, by \eqref{eq:=0}, $\sum_{j=1}^n y_j\,(x_j-x_{j-1})e^{itx_j}=0$ or, equivalently, 
\begin{equation*}%\label{eq:=0}
	\sum_{j\in J} y_j\,(x_j-x_{j-1}) e^{itx_j}=0, 
\end{equation*}
where 
\begin{equation*}
	J:=\big\{j\in\{1,\dots,n\}\colon x_j-x_{j-1}\ne0\big\}.
\end{equation*}
Note that the $x_j$'s for $j\in J$ are pairwise distinct -- because for any $j$ and $k$ in $J$ such that $j<k$ we have $x_j\le x_{k-1}<x_k$. 
Using now the textbook fact that exponential functions are linearly independent on any nonempty open interval (cf.\ e.g.\ Lemma~3.2 on page~92 in \cite{barbu} or a more general version for group characters \cite[Theorem~12, page~38]{artin_galois}), we conclude 
that $y_j\,(x_j-x_{j-1})=0$ for all $j\in J$ and hence for all $j\in\{1,\dots,n\}$, which completes the proof of Theorem~\ref{prop:PC}. 
\end{proof}

\begin{proof}[Proof of Theorem~\ref{prop:GM}]
Let $F$ be the cdf corresponding to the pdf $K$, so that 
\begin{equation*}
	F(x)=\int_{-\infty}^x K(u)\,du
\end{equation*}
for $x\in[-\infty,\infty]$. Also introduce 
\begin{equation*}
	(\De_i F)(x):=F(x-s_{i-1})-F(x-s_i) 
\end{equation*}
for all $x\in[-\infty,\infty]$ and all $i=1,\dots,n$ and
\begin{equation*}
	\De y_i:=y_i-y_{i-1}
\end{equation*}
for all $i=2,\dots,n$, so that 
\begin{equation*}
	y_i=y_1+\sum_{j=2}^i \De y_j
\end{equation*}
for all $i=1,\dots,n$; as usual, $\sum_{j=2}^1\dots:=0$. 
Then, recalling the definition \eqref{eq:GM} of the GM kernel estimator, for all real $x$ we have 
\begin{align*}%\label{eq:GM}
	\hf^\GM_K(x)&=\sum_{i=1}^n y_i\, (\De_i F)(x) \\ 	 
	&=\sum_{i=1}^n \Big(y_1+\sum_{j=2}^i \De y_j\Big) (\De_i F)(x) \\ 	 
	&=y_1 \sum_{i=1}^n (\De_i F)(x)
	+\sum_{i=1}^n \sum_{j=2}^i \De y_j\, (\De_i F)(x) \\ 	 
	&=y_1
	+\sum_{j=2}^n \De y_j\,\sum_{i=j}^n (\De_i F)(x) \\ 	 
	&=y_1
	+\sum_{j=2}^n \De y_j\,F(x-s_{j-1}), 	 
\end{align*}
taking into account that $s_0=-\infty$ and $s_n=\infty$, whereas $F(-\infty)=0$ and $F(\infty)=1$. 
Also, by \eqref{eq:y incr}, $\De y_j\ge0$ for all $j=2,\dots,n$. 
Now it is obvious that the function $\hf^\GM_K$ is nondecreasing. This completes the proof of Theorem~\ref{prop:GM}. 
\end{proof}

\bibliographystyle{amsplain}
%%%%\bibliography{citations.nodoi}
%%%
\bibliography{C:/Users/ipinelis/Documents/pCloudSync/mtu_pCloud_02-02-17/bib_files/citations01-09-20}
%\bibliography{P:/pCloudSync/mtu_pCloud_02-02-17/bib_files/citations01-09-20}

\end{document}